\documentclass[letterpaper, 10 pt, conference]{ieeeconf}
\IEEEoverridecommandlockouts
\usepackage{graphics} 
\usepackage{epsfig} 
\usepackage{amsmath} 
\usepackage{amssymb}  
\usepackage{bbm}
\usepackage{stmaryrd}
\usepackage{mathabx}
\usepackage{color}
\usepackage{tikz}
\usepackage{lipsum}
\usepackage{todonotes}
\usepackage{cite}

\usetikzlibrary{patterns,calc,angles,quotes,positioning,datavisualization}
\tikzset{
  treenode/.style = {shape=rectangle, rounded corners,
                     draw, align=center,
                     top color=white, bottom color=blue!20},
  root/.style     = {treenode, font=\Large, bottom color=red!30},
  env/.style      = {treenode, font=\ttfamily\normalsize},
  dummy/.style    = {circle,draw,minimum size = 0.6cm,pattern=crosshatch, pattern color = black!50},
  graph/.style    = {circle,draw,minimum size = 0.6cm,pattern=dots, pattern color = black!25},
  action/.style   = {circle,draw,minimum size = 0.6cm,pattern=crosshatch, pattern color = black!25},
}

\newtheorem{lemma}{Lemma}
\newtheorem{theorem}{Theorem}

\newtheorem{definition}{Definition}

\newtheorem{remark}{Remark}
\newtheorem{proposition}{Proposition}

\newcommand{\lap}[0]{{\mathcal{L}}}

\newcommand{\graph}[0]{{\mathcal{G}}}
\newcommand{\Graph}[0]{{\mathcal{G}}}

\newcommand{\Tr}{\mathrm{Tr}}

\newcommand{\Nodes}{\mathcal{N}}

\newcommand{\Edges}{\mathcal{E}}
\newcommand{\Weights}{\mathcal{W}}

\newcommand{\Htwo}{{\mathcal{H}_2}}

\usepackage{mathtools}
\usepackage{soul}
\usepackage{tikz}
\usepackage{color}
\usepackage{tcolorbox}
\usepackage{float}
\usepackage{subfig}
\usepackage[noend, noline]{algorithm2e}
\newlength\mylen

\makeatletter
\newcommand{\removelatexerror}{\let\@latex@error\@gobble}
\makeatother
\makeatletter
\renewcommand{\@algocf@capt@plain}{above}
\makeatother

\usepackage{tikzscale}
\mathtoolsset{showonlyrefs}
\usepackage{textcomp}

\title{\LARGE \bf
Efficient Computation of ${\cal H}_2$ Performance \\on Series-Parallel Networks
}

\author{Mathias Hudoba de Badyn and Mehran Mesbahi
\thanks{The research of the authors was supported by NSERC (ref \# CGSD2-502554-2017), the U.S. ARL and the U.S. ARO (contract \# W911NF-13-1-0340), and the U.S. AFOSR (funding \# FA9550-16-1-0022). }%
\thanks{The authors are with the William E. Boeing Department of Aeronautics and Astronautics at the University of Washington, Seattle, WA 98195 USA e-mails: \texttt{\{hudomath,mesbahi\}@uw.edu}.}%
}

\begin{document}

\maketitle
\thispagestyle{empty}
\pagestyle{empty}

\begin{abstract}
Series-parallel networks are a class of graphs on which many NP-hard problems have tractable solutions.
In this paper, we examine performance measures on leader-follower consensus on series-parallel networks.
We show that a distributed computation of the $\Htwo$ norm can be done efficiently on this system by exploiting a decomposition of the network into atomic elements and composition rules.
Lastly, we examine the problem of adaptively re-weighting the network to optimize the $\Htwo$ norm, and show that it can be done with similar complexity.
\end{abstract}


\section{Introduction}

Large-scale complex networks appear in many natural and synthetic systems, such as animal flocking \cite{Canizo2011}, robotic swarms~\cite{Hudobadebadyn2018}, oscillators \cite{Matheny2019}, opinion dynamics \cite{Hegselmann2002} and infrastructure networks \cite{Alemzadeh2018}.
Work on understanding the effect of the network structure on the behaviour of these systems has helped develop  many control and estimation algorithms.

An algorithm of particular interest is \emph{consensus}, a distributed information-sharing protocol over a network, with diverse applications such as Kalman filtering~\cite{Olfati-Saber2005,HudobadebadynFillt2017}, multi-agent systems~\cite{Chen2013a}, swarm deployment~\cite{Hudobadebadyn2018}, and robotics~\cite{Joordens2009}.
The controllability of consensus is related to symmetries of the graph; automorphisms fixing the control input yield uncontrollable modes~\cite{Rahmani2009a,Chapman2015a,Alemzadeh2017}.
This argument was extended to nonlinear variants of the protocol~\cite{Aguilar2014}, and graphs admitting both positive and antagonistic interactions~\cite{Alemzadeh2017,Hudobadebadyn2017,Altafini2013}.

One concern is how the algorithm performs with respect to various measures, such as entropy~\cite{HudobaDeBadyn2015a,Siami2017}, or disturbance rejection in the form of the $\Htwo$ norm~\cite{Siami2014a,Chapman2015,Chapman2013a}.
Such measures provide a metric for the underlying dynamical system to be able to autonomously modify the topology of the network to adapt to antagonistic influences~\cite{Chapman2013a,Chapman2015}.

The notions of performance and control of consensus all rely on the underlying topology of the network.
Several famous models for graph design are scale-free networks~\cite{Albert2002} and the Erdos-Renyi random graph model~\cite{Mesbahi2010}.
Graph growth models have been considered both for control~\cite{Hudobadebadyn2016} and performance~\cite{Siami2016}.
A more axiomatic approach is to take small, atomic elements and build graphs out of these elements, or to use graph-growing operations that preserve properties like controllability~\cite{Hudobadebadyn2016,chapman2014cart,Alemzadeh2018}.

In this paper, we seek to understand how system-theoretic measures can be computed from atomic elements that build up the graph.
We approach the performance problem on leader-follower consensus in this spirit by considering \emph{series-parallel networks}.
These networks can be decomposed into a series of atomic elements and composition operations in linear~\cite{Valdes1979} and sublinear~\cite{Eppstein1992} time.
Their key property is that many NP-hard problems on general classes of graphs become linear on series-parallel graphs~\cite{Takamizawa1986}.

The contributions and organization of the paper are as follows.
By exploiting the structure of series-parallel networks, we are able to present a way of computing the $\Htwo$ norm of a leader-follower consensus network in best-case $\mathcal{O}(|\mathcal{R}|\log|\Nodes|)$ (worst case $\mathcal{O}(|\mathcal{R}||\Nodes|)$) complexity, where $|\mathcal{R}|, |\Nodes|$ are the number of leaders and followers, respectively.
We also provide a method for adaptively re-weighting the network to optimize $\Htwo$ performance that utilizes computations of similar complexity by invoking an electrical network analogy of the consensus dynamics.
In \S\ref{sec:math-prel}, we outline the notation used in the paper.
The problem statement, background on series-parallel networks, and the electrical analogy of consensus are in \S\ref{sec:graph-compositions}.
Our main results on efficient $\Htwo$ computation and adaptive re-weighting are presented in \S\ref{sec:syst-theor-comp}, with conclusions in \S\ref{sec:conclusion}.

\section{Mathematical Preliminaries}
\label{sec:math-prel}

\subsection{Graph and Edge Laplacians}

The vector $e_i\in\mathbb{R}^n$ has `1' in its $i$th entry, and zeros elsewhere, and we denote the set $\mathbb{R}_{++}=\{x\in \mathbb{R}:x>0\}$.
The \emph{parallel addition} of two scalars $x,y\neq 0$ is defined as $x:y \triangleq (x^{-1} + y^{-1})^{-1}$.
A \emph{graph} $\Graph$ is a triple of sets $(\Nodes,\Edges,\Weights)$ where $\Nodes$ is a set of \emph{nodes}, $\Edges\subseteq \Nodes^2$ is a set of \emph{edges} denoting pairwise connections between nodes, and $\Weights \in \mathbb{R}^{|\Edges|}$ is a set of \emph{weights} on the edges.
We use the standard graph theoretic notation regarding graphs, and adjacency, degree, incidence and Laplacian matrices~\cite{Mesbahi2010}.
Given two graphs $\Graph_1,\Graph_2$, one may combine them through the operation of \emph{identifying nodes}.
Suppose $s_1\in\Nodes_1,s_2\in\Nodes_2$.
Then, we write $\Graph \leftarrow s_1\sim s_2$ to indicate that $\Nodes(\graph) = (\Nodes_1\setminus \{s_1\})\cup(\Nodes_2\setminus\{s_2\})\cup \{s\}$, where if $s_1i\in\Edges(\Graph_1)$ or $s_2i\in \Edges(\Graph_2)$ implies that $si\in\Edges(\Graph)$ (similarly, if considering directed edges, we have that if $is_1\in\Edges(\Graph_1)$ or $is_2\in \Edges(\Graph_2)$ implies that $is\in\Edges(\Graph)$).
Such an operation is depicted in Figure~\ref{fig:leaderFollower}: the two orange squares in the oval denoted by $\mathcal{R}$ are identified to the node $l$ in the bottom graph.




A \emph{tree} is a connected graph with no cycles, and a \emph{leaf} is designated as a node of degree 1.
A \emph{binary tree} is a tree where one node is designated as the \emph{root}, and all the nodes of $\mathcal{T}$ are either leaves or \emph{parents}.
Each parent in a binary tree can have at most two \emph{children} and one parent, except the root which has no parents.
The \emph{height} $h$ of a binary tree is the length of the longest path from the root to a leaf.
A \emph{complete} (sometimes called \emph{full}) binary tree is one which each node has either zero or two children.

\section{Network Models}
\label{sec:graph-compositions}
\subsection{Problem Statement}
\label{sec:problem-statement}

In this paper, we examine the leader-follower consensus problem.
Consider a connected weighted graph $\mathcal{G} = (\Nodes,\Edges,\Weights)$ with Laplacian $\lap$.
Denote a set of \emph{leaders} $\mathcal{R}\subset \Nodes$, and a set of \emph{followers} $\Nodes\setminus\mathcal{R}$.
Further suppose that each leader is connected to a single unique node in $\Nodes\setminus\mathcal{R}$, called the \emph{input} nodes, denoted by $R$ (see Figure~\ref{fig:leaderFollower} for a schematic of the setup).
Denote the set of input nodes by $R$.
Then, the graph Laplacian of $\graph$ can be partitioned as follows:
\begin{figure}
  \centering
  \raisebox{-.5\height}{%
    \includegraphics[width=0.4\columnwidth]{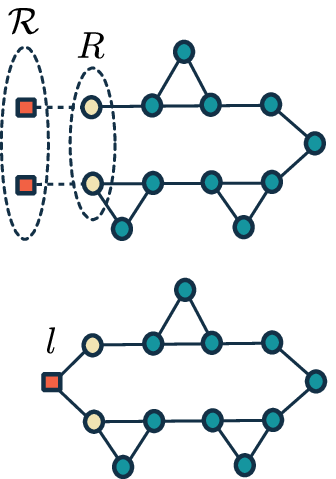}%
  }\qquad
  \raisebox{-.5\height}{%
    \includegraphics[width=0.4\columnwidth]{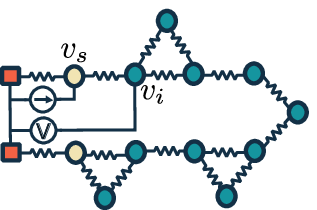}%
  }
  \caption{Top: leader-follower network setup. Right: electrical ``grounding'' of the leader set $\mathcal{R}$, and current vector $e_s$ injected into the network. $y_s^i$ is the voltage dropped from $v_i$ to $\mathcal{R}$. Bottom: Identification of grounded leader set into a single node.}
  \label{fig:leaderFollower}
\end{figure}
\begin{align}
  \lap_\graph = \left[
  \begin{array}{c|c}
    \mathcal{L}_{\mathcal{R}}  & -B^TW_{\mathcal{R}} \\\hline
    -W_{\mathcal{R}}B & \mathcal{L}_{\mathcal{G}(\Nodes\setminus\mathcal{R})} + \sum_{i\in {R}} e_ie_i^T
  \end{array}\right],\label{eq:4}
\end{align}
where $W_{\mathcal{R}}$ is a diagonal matrix containing the weights of the edges connecting ${\mathcal{R}}$ to $R$.
Suppose we can observe individual follower nodes.
Therefore, the control and observation matrices are given by $B^T = [e_{i_1}~\dots~e_{i_m}]$ and $C = [e_{j_1}~\dots~e_{j_o}]$,
where $R = \{i_1,\dots,i_m\}$ are the nodes attached to leaders, and $ \{j_1,\dots,j_o\}$ are the nodes under observation.
Note that the form of these matrices assumes that each leader is attached to a single, unique follower.
Next, recall that one can write the graph Laplacian in terms of the incidence and weight matrices as
$
 \mathcal{L}_{\mathcal{G}(\Nodes\setminus\mathcal{R})}= EWE^T = \sum_{e\in\Edges} w_{e}a_{e}a_{e}^T.
$
Then, the corresponding leader-follower consensus dynamics, with $W_{\mathcal{R}}=I$, are given by 
\begin{align}
  \dot x &= - \left(  \mathcal{L}_{\mathcal{G}(\Nodes\setminus\mathcal{R})} + \sum_{i\in R} e_ie_i^T \right)x + Bu,~
  \tilde{y} = Cx.\label{eq:11}
\end{align}
The matrix $A \triangleq - (  \mathcal{L}_{\mathcal{G}(\Nodes\setminus\mathcal{R})} + \sum_{i\in R} e_ie_i^T )$ is negative definite if the underlying graph is connected, and thus invertible.

\subsection{$\mathcal{H}_2$ Norm}

For a linear system and corresponding transfer function,
\begin{align}
  \dot x &= Ax + Bu,~y= Cx,
  G(s) = C(sI-A)^{-1}B,
\end{align}
the $\mathcal{H}_2$ norm is defined as 
\begin{align}
  \|G(s)\|_2^2 = \dfrac{1}{2\pi} \int_{-\infty}^\infty \Tr \left[G(j\omega)^TG(j\omega) \right] d\omega,
\end{align}
and measures the root-mean-square of the inpulse response of the system, or equivalently the steady-state covariance of the system under zero-mean, unit-covariance white noise inputs.
The $\mathcal{H}_2$ norm of the leader-follower consensus dynamics is discussed in \S\ref{sec:syst-theor-comp}.

\subsection{Electrical Network Models}

An interesting perspective on consensus networks comes from viewing the underlying graph as an electrical network of resistors~\cite{Chapman2013a,Barooah2008}.
In the leader-follower setup of \S\ref{sec:problem-statement}, consider the graph $\graph = (\Nodes,\Edges,\Weights)$.
For each edge $e=ij\in\Edges$, consider placing a resistor between nodes $i$ and $j$ with conductance $w_{ij}$ (resistance $w_{ij}^{-1}$).
Then, the \emph{effective resistance} between arbitrary nodes $k,l\in \Nodes$ is given by 
\begin{align}
  R_{kl} = (e_k - e_l) \lap^\dag (e_k-e_l).\label{eq:9}
\end{align}

Alternately, consider `grounding' all the leader nodes $r\in\mathcal{R}$, i.e., identifying them all as one node, or electrically connecting them together by wires.
Then, the diagonal entries of the matrix $A^{-1}$, with $A$ from~\eqref{eq:11}, yield the effective resistances between the $i$th node in $\Nodes\setminus\mathcal{R}$ to the leader node set, i.e. $[A^{-1}]_{ii}$ is the the effective resistance from $i$ to $\mathcal{R}$.

If $x\in\mathbb{R}^{|\Nodes|}$ denotes a vector of current injected into the nodes $\Nodes\setminus\mathcal{R}$ from $\mathcal{R}$, then the quantity $[A^{-1}x]_i$ is the \emph{voltage drop} from node $v_i$ to the grounded leader node set.
If $x = e_s$, then this corresponds to a current of 1A injected into the node $v_s\in R$, see Figure~\ref{fig:leaderFollower} for a schematic of this setup.
In this case, we write $y_i^s \triangleq [A^{-1}e_s]_i$ as the voltage drop of node $i$.

\subsection{Series-Parallel Graph Models}

In this paper, we consider the class of graphs known as \emph{series-parallel graphs}.
Given a series-parallel graph there exist extremely efficient ($\mathcal{O}(|\Nodes| + |\Edges|)$, $\mathcal{O}(\log |\Nodes|)$ and $\mathcal{O}(\log^2 |\Nodes|)$) algorithms that decompose the graph into atomic structures and simple composition operations on them~\cite{Valdes1979,Eppstein1992,He1987}.
This decomposition allows for efficient solutions to problems that are otherwise NP-hard on general classes of graphs~\cite{Takamizawa1986}.

\begin{definition}[Two-Terminal Series-Parallel Graphs]
  \label{def:ttsp}
A directed acyclic graph is called \emph{two-terminal series-parallel TTSP} if  it can be defined recursively as follows:
\begin{enumerate}
  \item The graph defined by two vertices connected by an edge (a \emph{1-path}) is a TTSP graph, where one node is labeled the \emph{source}, and the other the \emph{sink}.
  \item If $\graph_1=(\Nodes_1,~\Edges_1)$ and $\graph_2=(\Nodes_2,~\Edges_2)$ are TTSP where $\mathcal{S}_i = \{s_i\},\mathcal{T}_i=\{t_i\}$ are the unique source and sink of $\graph_i$, then the following operations produce TTSP graphs: 
    \begin{enumerate}
      \item \textbf{Parallel Addition:} $\graph_p \leftarrow s_1\sim s_2,~t_1\sim t_2$.
      \item \textbf{Series Addition:} $\graph_s \leftarrow t_1\sim s_2  $.
    \end{enumerate}
\end{enumerate}  
Denote the parallel join of $\graph_1$ and $\graph_2$ as $\graph_1\oslash\graph_2$, and the corresponding series join as $\graph_1\odot\graph_2$.
\end{definition}

\begin{figure}
  \centering
  \includegraphics[width=0.6\columnwidth]{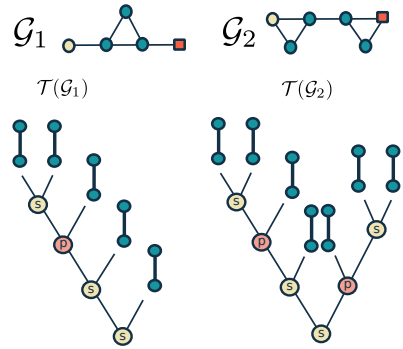}
  \caption{Decomposition trees of two graphs.}
  \label{fig:decompTree}
\end{figure}

The two recursive operations defining the TTSP graph model allow for a simple constructive approach for defining graphs from atomic elements.
Indeed, efficient algorithms exist that decompose TTSP graphs into a \emph{decomposition tree} with the following structure~\cite{Valdes1979,Eppstein1992,He1987}. 
\begin{definition}[TTSP Decomposition Tree]
  A \emph{TTSP decomposition tree} of a TTSP graph $\graph$ is a binary tree $\mathcal{T}(\graph)$ with the following properties: 
  \begin{enumerate}
    \item $\mathcal{T}$ is a complete (sometimes called \emph{full}) binary tree, in that every node has either 2 or 0 children.
    \item Every leaf of $\mathcal{T}$ corresponds to a 1-path.
    \item Every parent of $\mathcal{T}$ corresponds to either a series or parallel addition operation from Definition~\ref{def:ttsp} on its children.
  \end{enumerate}
\end{definition}
\begin{remark}
  In general, one can ignore the assumption of directed edges in Definition~\ref{def:ttsp} and consider undirected TTSP graphs, see~\cite{Eppstein1992}.
\end{remark}
Examples of TTSP graphs and their decomposition trees are shown in Figure~\ref{fig:decompTree}.
The graph is constructed by a reverse breadth-first-search: the deepest leaves combine each other by the series or parallel join designated by their parent.
At each layer of the decomposition tree, each operation can be performed independently of the other nodes in the layer.
Hence, the complexity of the reconstruction operation is limited by the height of the decomposition tree; it is this key insight that allows the efficient computations of the system-theoretic measures in the remainder of the paper.
First, in the following proposition, we quantify the  height of the tree in terms of the size of the resulting graph.

\begin{proposition}[Properties of $\mathcal{T}(\graph)$]
  \label{prop:tree}
  Let $\graph$ be a TTSP graph with $N$ nodes constructed from $l$ 1-paths with $p$ parallel joins and $s$ series joins.
  Then, 
  \begin{enumerate}
    \item $\graph$ has $N=2l-2p-s$ nodes and $E=l$ edges.
    \item The decomposition tree $\mathcal{T}(\graph)$ of $\graph$ has $n = 2l-1$ nodes
    \item The height $h$ of $\mathcal{T}(\graph)$ is bounded by 
      \begin{align}
        \log_2\left(N+2p-s\right) \leq& h \leq \frac{N+2p+s}{2} -1\label{eq:3}\\
        \log_2(E) \leq& h \leq E-1\label{eq:7}
      \end{align}
  \end{enumerate}
\end{proposition}
\begin{proof}
  BEach parent of $\mathcal{T}(\graph)$ has either 2 or 0 children.
  If it has zero children, it is a leaf and therefore corresponds to a 1-path, which adds one edges to $\graph$.
  It follows that $E \triangleq |\Edges| = l$.
  
  Each leaf also adds two nodes to $\graph$; each series join identifies a pair of nodes, and each parallel join identifies two pairs of nodes.
  Hence, 
  $
    N \triangleq |\Nodes| = 2l -2p -s.
  $
  A complete binary tree has $n=2l-1$ nodes, and so
  $
    n = N+2p +s -1.
  $
  Furthermore, the number of nodes $n$ of $\mathcal{T}(\graph)$ is bounded by its height by the relation
  $
    2h+1 \leq n \leq 2^{h+1} -1,
  $
  leading to  
  \begin{align}
    \log_2(n+1)-1 \leq h \leq \frac{n-1}{2}.\label{eq:5}
  \end{align}
  Substituting $n$ and $l$ into \eqref{eq:5} yields the result.
  
\end{proof}

\section{System-Theoretic Computations on Series-Parallel Graphs}
\label{sec:syst-theor-comp}

\subsection{Efficient Computations on Series-Parallel Graphs}

We now discuss how series-parallel graphs can be used to simplify computations.
Some quantities are `hard' to compute on large graphs, but they may be easy to compute over small atomic elements, and the quantities may propagate across graph compositions (i.e., series or parallel connections) by known formulae.
Given a decomposition tree, this operation takes $\mathcal{O}(h)$ time.

A natural use for series-parallel graphs is in electrical networks, as resistances (conductances) add over a series (parallel) connection of resistors.
Na\"{i}ve methods for computing resistances over graphs are expensive, such as in~\eqref{eq:9} which has complexity $\mathcal{O}(|\Nodes|^a)$, $a>2$.

Recall that one can construct the graph by moving up the tree in a reverse breadth-first-search order and joining the graphs at each leaf via the operation prescribed by their parent.
Similarly, given a TTSP network of resistors (edges with weights corresponding to conductances) and its decomposition tree, one can start at the leaves (corresponding to 1-paths, i.e., individual resistors), and either add resistances together (if the join is series), or add their reciprocals (if the join is parallel).
By performing this computation up the tree to the root, one yields the effective resistance across the TTSP graph from source to sink.

\begin{figure}
  \centering
  \includegraphics[width=0.8\columnwidth]{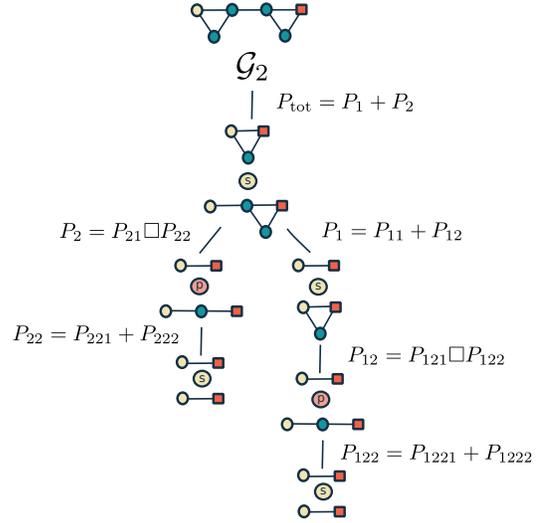}
  \caption{Nested infimal convolution/sum computations over a series-parallel graph. Beige circles (always left) denote sources, red squares (always right) denote sinks at each step.}
  \label{fig:computation}
\end{figure}

In this manner, one can perform additional electrical calculations.
Suppose that  a current is applied across the TTSP graph, from source to sink.
Given measurements of currents across the resistors, it is possible to compute the total power dissipated across the graph.
Over a series connection of resistors $r_1,r_2$, the power dissipated is additive: $P_{s} = P_1 + P_2 = i^2(r_1+r_2)$.
Over a parallel connection of resistors, the total power dissipated is minimized over the resistors, subject to conservation of current: 
\begin{align}
  P_p(i) = 
  \begin{array}{lc}
    \min & i_1^2 r_1 + i_2^2 r_2\\
    \text{s.t.} & i_1+ i_2 = i 
  \end{array}\triangleq P_1 \square P_2.\label{eq:10}
\end{align}
The operation denoted $P_1 \square P_2$ is denoted the \emph{infimal convolution}, and an example of the computation of powers over a TTSP graph is shown in Figure~\ref{fig:computation}.

\subsection{Noise Rejection and Adaptive Weight Design}

It is often the case that one wishes to adapt the network in order to reject noise.
For example, in a swarm of UAVs performing consensus on heading, it is undesirable for the swarm to be influenced by wind gusts.
In cyberphysical systems, noise may be injected by an adversary wishing to disrupt a process on the network.
In this section, we discuss an autonomous protocol that allows networks to quickly adapt their communication edge weights in order to minimize the response to such disturbances.

Consider the leader-follower consensus dynamics setup from \S\ref{sec:problem-statement}.
When considering the performance of all nodes ($C=I$), the $\Htwo$ performance is given by 
$ (\Htwo^{\graph, B})^2 = - {2}^{-1} \Tr (B^T A^{-1} B)
$
~\cite{Chapman2015}.

Consider the task of re-assigning positive weights to the edges of the network to minimize the $\Htwo$ norm.
This can be done via a gradient-descent method solving the problem $\Sigma$:
\begin{align}
\Sigma:= \left\{ \begin{array}{ll}
    \min_W &  f(W) \triangleq \left(\Htwo^{\graph, B}\right)^2 + \dfrac{h}{2}\|W\|^2\\
         & -\frac{1}{2}\left(B^TA(W)^{-1}B\right) + \dfrac{h}{2}\|W\|^2\\
    \text{s.t.} & \mathrm{diag}(W)\in \mathbb{R}_{++}^{|\Edges|}.
  \end{array}\right\}\label{eq:prob}
\end{align}
$\Sigma$  has several key features.
First, note that on the cone of positive-definite matrices, the map $A\to A^{-1}$ is matrix convex.
Second, $A(W)$ is linear in the diagonal matrix $W$ since 
$
  A(W) = - \sum_{e\in\Edges} w_{e}a_{e}a_{e}^T - BB^T.
$
Therefore, the objective function $f(W)$ in Problem~\eqref{eq:prob} is strongly convex.

Following \cite{Chapman2015} and denoting $e:=ij$, note that the derivative of the objective function is given by 
$
 \frac{\partial f(W)}{\partial w_{ij}} =  -\frac{1}{2} \sum_{s\in R} ( y_i^s - y_j^s)^2 + hw_{ij},
$
where the \emph{voltage drop} of node $i$ with respect to source  node $s\in R$ is precisely $y_i^s\triangleq \left[A(W)^{-1}e_s\right]_i$.
Thus, each edge in the network updates its weight according to the dynamics given by the gradient update
\begin{align}
  w_{ij}^{t+1} = \left(1-\frac{1}{\sqrt{t}}\right)w_{ij}^t + \dfrac{1}{2\sqrt{t}} \sum_{s\in R} \left(y_i^s - y_j^s\right)^2.\label{eq:8}
\end{align}
A centralized algorithm and a decentralized conjugate gradient algorithm for solving Problem~\eqref{eq:prob} was presented in \cite{Chapman2015}.
The complexity in solving this problem is the computation of the voltage drops $y_i^s,y_j^s$ for each edge $\{i,j\}\in \Edges$, which takes linear time (in $|\Nodes|$ steps) for each edge.

We now present a decentralized algorithm for computing this quantity on a certain class of two-terminal series-parallel graphs, and a characterization of its complexity.
In particular, for every input node $s\in R$, the graph has to be TTSP with $s$ as the sink, and a node representing the grounded leader set as the source.
\begin{definition}[All-Input TTSP Graphs]
  Consider a graph $\graph$ in the setup of the leader-follower consensus dynamics.
  Identify each leader node $i_1,i_2,\dots,i_r\in \mathcal{R}$ as one node $l$ connected to all source nodes $s\in R$, as depicted in Figure~\ref{fig:leaderFollower}.
  The graph $\graph$ is called an \emph{all-input two-terminal series-parallel graph} if, for all source nodes $s\in R$, $\graph$ is TTSP with $s$ as the source and $l$ as the sink.
\end{definition}
See Figure~\ref{fig:examples} for an  example of an all-input TTSP graph.

Informally, the algorithm is as follows.
For each source node $s$, the algorithm utilizes the decomposition tree of $\graph$ with $s$ as the source and the grounded leader set as the sink; this is why $\graph$ needs to be TTSP with respect to all source nodes.
The voltage drops can be computed from resistances and currents across each join, and the currents can be extracted from the power dissipated across each join.
Hence, the effective resistances are computed first, as in Algorithm~\ref{alg:1}.
Starting from the root of the decomposition tree, the currents at each join can be computed by Eq.~\eqref{eq:10}, as in Algorithm~\ref{alg:2} and depicted in Figure~\ref{fig:computation}.
Finally, the voltage drops over each branch are computed starting from the leaves of the decomposition tree, as seen in Algorithm~\ref{alg:3}.

\removelatexerror
\begin{algorithm}[H]
  \label{alg:1}
  \SetAlgoNoLine
 \caption{Effective Resistance over TTSP Graph}
\SetAlgoLined
\KwIn{$\mathcal{T}(\graph),~\Weights(\graph)$}
\KwResult{ Effective resistances $\rho(\graph)$ over $\mathcal{T}(\graph)$ }
 \For{each leaf of $\mathcal{T}(\graph)$}{
  Output $R_{\text{eff}} = w_e^{-1}$ to parent\;
 }
 \For{each parent $j$ of $\mathcal{T}(\graph)$}{
  \eIf{received $R_{\text{eff}_i}$ from both children $i=1,2$}{
    \eIf{$j$ is a series join}{
       Output $R_{\text{eff}} = R_{\text{eff}_1}+R_{\text{eff}_2} $ to parent\;
    }{
       Output $R_{\text{eff}} = R_{\text{eff}_1}:R_{\text{eff}_2} $ to parent\;
    }
   }{
   wait\;
  }
 }
\end{algorithm}
\vspace{-0.5mm}

\removelatexerror
\begin{algorithm}[H]
  \label{alg:2}
 \caption{Branch Currents over TTSP Graph}
\SetAlgoLined
\KwIn{$\mathcal{T}(\graph),~\rho(\graph)$}
\KwResult{ Currents $\mathcal{I}(\graph)$ over $\mathcal{T}(\graph)$ }
 \For{each parent $j$ of $\mathcal{T}(\graph)$}{
   \uIf{$j$ is root}{
         \eIf{$j$ is a series join}{
           Output $i_{\text{out}} = 1$ to children\;
         }{
           Output $(i_1,i_2) = \arg P_1\square P_2$ to children\;
         }
   }
   \uElseIf{received $i_{\text{in}}$ from parent}{
         \eIf{$j$ is a series join}{
           Output $i_{\text{out}} = i_{\text{in}}$ to children\;
         }{
           Output $(i_1,i_2) = \arg P_1\square P_2$ to children\;
         }
   }
   \Else{
     wait\;
   }
 }

\end{algorithm}

\removelatexerror
\begin{algorithm}[H]
  \label{alg:3}
 \caption{Voltage Drops over TTSP Graph}
\SetAlgoLined
\KwIn{$\mathcal{T}(\graph),~\rho(\graph),~\mathcal{I}(\graph)$}
\KwResult{Voltage drops $y_i^s$ over $\mathcal{T}(\graph)$ }
 \For{each leaf of $\mathcal{T}(\graph)$}{
  Output $v_e = i_ew^{-1}_e$ to parent\;
 }
 \For{each parent $j$ of $\mathcal{T}(\graph)$}{
  \eIf{received $v_{e_i}$ from both children $i=1,2$}{
    \eIf{$j$ is a series join}{
       Output $v_{\text{out}} = v_{\text{in}_1}+v_{\text{in}_2} $ to parent\;
    }{
       Output $v_{\text{out}} = v_{\text{in}_1} = v_{\text{in}_2} $ to parent\;
    }
   }{
   wait\;
  }
 }
\end{algorithm}

\begin{figure} 
  \centering
  \includegraphics[width=0.7\columnwidth]{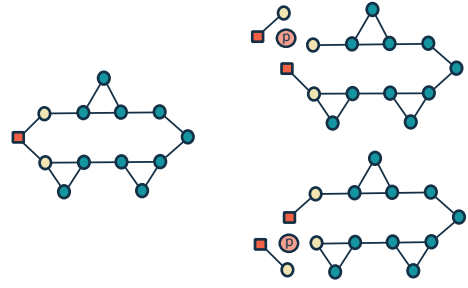}
  \caption{Left: An all-input TTSP graph. Right: Parallel joins across $R$ to $\mathcal{R}$ used to compute $y_s^s$.}
  \label{fig:examples} 
\end{figure}

We have the following result. 
\begin{theorem}
  \label{thm:complex}
  Consider the calculation of the voltage drops $y_i^s$ in the gradient update scheme of \eqref{eq:8}.
  The  best-case and worst-case complexity of this computation is $\mathcal{O}(|\mathcal{R}| \log |\Nodes|)$, and  $\mathcal{O}(|\mathcal{R}||\Nodes|)$, respectively.
\end{theorem}
\begin{proof}
 Algorithms 1,2 and 3 are computations done on a binary tree of height $h$, with each layer's computation done in parallel.
 Hence, by Proposition~\ref{prop:tree}, the complexity is $\mathcal{O}( \log |\Nodes|)$, and  $\mathcal{O}(|\Nodes|)$ for best and worst-case, respectively.
 The 3 algorithms are called $|\mathcal{R}|$ times.
\end{proof}

\subsection{Synthesis of $\Htwo$-Optimal Networks}

The algorithm described above can also be used to compute the $\Htwo$ norm.
First, let's examine the structure of the $\Htwo$ norm in the context of the leader-follower consensus setup.
\begin{lemma}
\label{lem:4}
  Consider the leader-follower consensus setup from \S\ref{sec:problem-statement}, with $C=I$.
  Then, 
  \begin{align}
    \left(\Htwo^{\graph, B}\right)^2 &= - \dfrac{1}{2} \sum_{s\in R} \left[ A^{-1} e_s \right]_s = -\dfrac{1}{2}\sum_{s\in R} y_s^s = \dfrac{1}{2}\sum_{s\in R} \rho^{\text{eff}}_{s,\mathcal{R}},
  \end{align}
where $\rho^{\text{eff}}_{s,\mathcal{R}}$ denotes the effective resistance from $s$ to the grounded leader node set $\mathcal{R}$.
\end{lemma}
\begin{proof}
  A calculation suffices: 
  \begin{align}
&\big(\Htwo^{\graph, B}\big)^2 = - \dfrac{1}{2} \Tr\big(B^T A^{-1} B\big) = - \dfrac{1}{2} \Tr\big(\sum_{i,s \in R} e_i A^{-1}_{is} e_s^T \big)\\
                  &=  - \dfrac{1}{2} \sum_{s\in R} \left[ A^{-1} e_s \right]_s
                  = - \dfrac{1}{2}\sum_{s\in R} y_s^s = \dfrac{1}{2}\sum_{s\in R} \rho^{\text{eff}}_{s,\mathcal{R}}\label{eq:13}.
  \end{align}
\end{proof}
Each quantity $y_s^s$ of the left side of the sum in~\eqref{eq:13} is the voltage drop from source node $s\in R$ to the grounded leader node set $\mathcal{R}$.
This is precisely the voltage dropped over the last parallel join of the series-parallel decomposition of an all-input TTSP graph; depicted in Figure~\ref{fig:examples}.
We can utilize this  observation to efficiently compute the $\Htwo$ norm \emph{a priori} knowing only the weights of the edges and the decomposition tree of $\graph$.

We use the following setup.
Consider a TTSP graph $\graph$ with source node $s$ and sink node $t$.
Ground the source node $s$, and consider the grounded Laplacian $A$ with respect to the grounded source $s$.
This is a leader-follower system with a single leader.

Note that the parallel join depicted in Figure~\ref{fig:examples} makes one of the terminals of the resulting graph an element of $\mathcal{R}$, and the other terminal (the sink) an element of $R$.
The control matrix of the leader-follower consensus problem corresponds to exactly those elements in $R$, which is the `sink' of the TTSP graph used in that computation.
Therefore, our choice of $B$ selects the sink vector $t$; hence $B = e_t$.

We now proceed in two steps.
First, we need a lemma that effectively says that for an arbitrary TTSP graph, there exists an equivalent 1-path TTSP graph with the same effective resistance.
Then, any composition rule on arbitrary TTSP graphs can be reduced to a composition on the equivalent 1-paths, simplifying analysis.
Afterward, we discuss a series-parallel calculation of the $\Htwo$ performance.

\begin{lemma}
  \label{lem:5}
  Consider two graphs: an arbitrary TTSP graph $\graph_1$ with source $s_1$ and sink $t_1$ with effective resistance $\rho^{\text{eff}}_{s_1,t_1}$, and a 1-path TTSP graph $\graph_2$ with source $s_2$ and sink $t_2$ with effective resistance $\rho^{\text{eff}}_{s_2,t_2}$.
  Let their respective control matrices be $B_1 = e_{t_1}$ and $B_2=e_{t_2}$.
  Further suppose that 
$
    \rho^{\text{eff}}_{s_1,t_1}=\rho^{\text{eff}}_{s_2,t_2}.
$
  Then, 
$
    (\Htwo^1)^2 = (\Htwo^2)^2.
$
\end{lemma}
\begin{proof}
  The setup of the graphs is a leader-follower consensus with grounded (leader) nodes $s_1,s_2$.
  Therefore, we can invoke Lemma~\ref{lem:4}.
  Using Lemma~\ref{lem:4}, denoting the graphs' respective Dirichlet Laplacians as $A_1,A_2$ we can compute: 
  \begin{align}
     & (\Htwo^1)^2 
                 = \frac{1}{2}\Tr\left[e_{t_1}^T\left[A_1\right]^{-1}e_{t_1}\right]
                 = \frac{1}{2} \rho^{\text{eff}}_{s_1,t_1}     \\           &= \frac{1}{2} \rho^{\text{eff}}_{s_2,t_2}
                 = \frac{1}{2}\Tr\left[e_{t_2}^T\left[A_2\right]^{-1}e_{t_2}\right]                =  (\Htwo^2)^2.
  \end{align}
\end{proof}
Lemma~\ref{lem:5} will allow us to reduce the computation of the $\Htwo$ norms of a composite TTSP graph to the computation of $\Htwo$ norms of an equivalent 1-path. 
This allows us to prove the following theorem.
\begin{theorem}
  \label{thr:3}
  Consider two graphs: an arbitrary TTSP graph $\graph_1$ with source $s_1$ and sink $t_1$, and a second arbitrary TTSP graph $\graph_2$ with source $s_2$ and sink $t_2$. 
  Let the Dirichlet Laplacian of $\graph_i$ grounded with respect to its source $s_i$ be $A_i$, and its control matrix be $B_i = e_{t_i}$.
  Hence its $\Htwo$ norm is given by 
 $
    (\Htwo^{\graph_i})^2 = -\frac{1}{2}\Tr[B_i^T \left[A_i\right]^{-1} B_i ].
 $
  Then, 
  \begin{align}
    \left(\Htwo^{\graph_1\odot\graph_2}\right)^2   &= \left(\Htwo^{\graph_1}\right)^2 + \left(\Htwo^{\graph_2}\right)^2\label{eq:15} \\
    \left(\Htwo^{\graph_1\oslash\graph_2}\right)^2 &= \left(\Htwo^{\graph_1}\right)^2 :\left(\Htwo^{\graph_2}\right)^2 .\label{eq:14}
  \end{align}
\end{theorem}

\begin{figure} 
    \centering
  \subfloat[Series join of 1-paths]{%
       \includegraphics[width=0.4\columnwidth]{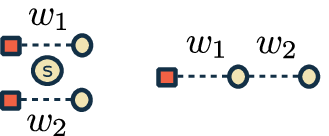}}
    \label{join1}~~
  \subfloat[Parallel join of 1-paths]{%
        \includegraphics[width=0.4\columnwidth]{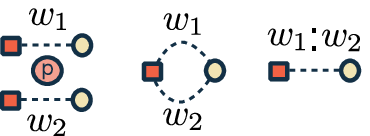}}
    \label{join2}
  \caption{Series and parallel joins of 1-paths with arbitrary weights}
  \label{fig:simplejoin}
\end{figure}

\begin{proof}
  By Lemma~\ref{lem:4}, the $\Htwo$ norm of an arbitrary graph can be computed from an equivalent 1-path.
  Hence, we need to show~\eqref{eq:15} and~\eqref{eq:14} for series and parallel joins of 1-paths.
  
  Consider the 1-paths in Figure~\ref{fig:simplejoin}, with weights $w_1$ and $w_2$ between the sources (square nodes) and sinks (circular nodes).
  The Dirichlet Laplacians of both with respect to the grounded source nodes are simply 
$
    \lap_{\graph_1} = \left[ w_1 \right],~\lap_{\graph_2} = \left[ w_2 \right],
$
  and their respective control matrices are $B_{\graph_1}= B_{\graph_2} = 1$
  and $\Htwo$ norms are 
 $
    (\Htwo^{\graph_1})^2 = \frac{1}{2}w_1^{-1}$, $(\Htwo^{\graph_2})^2 = \frac{1}{2} w_2^{-1}.
 $

  Similarly, the Laplacians of the series and parallel joins in Figure~\ref{fig:simplejoin} are,
  \begin{align}
     \lap_{\graph_1 \oslash\graph_2} = \left[ w_1+w_2\right] ,~\lap_{\graph_1 \odot\graph_2} = 
    \begin{bmatrix}
      w_1+w_2 & -w_2\\
      -w_2 & w_2
    \end{bmatrix}.
  \end{align}
  The control matrices are 
$
    B_{\graph_1 \oslash\graph_2} = 1$, $B_{\graph_1 \odot\graph_2} = 
[
      0 ~ 1
].
$
  Therefore, the $\Htwo$ norm in the series join case is 
  \begin{align}
    \left(\Htwo^{\graph_1\odot\graph_2}\right)^2   &= \frac{1}{2} \Tr\left[ 
                                                     \begin{bmatrix}
                                                       0 & 1 
                                                     \end{bmatrix}
                                                           \begin{bmatrix}
                                                             w_1+w_2 & -w_2\\
                                                             -w_2 & w_2      
                                                           \end{bmatrix}^{-1}
                                                     \begin{bmatrix}
                                                       0 \\ 1 
                                                     \end{bmatrix}\right]\\
                                                   &= \frac{1}{2}\left[ \frac{1}{w_1} + \frac{1}{w_2} \right]
                                                   = \left(\Htwo^{\graph_1}\right)^2 + \left(\Htwo^{\graph_2}\right)^2.
  \end{align}
  Similarly, the $\Htwo$ norm in the parallel join case is 
  \begin{align}
    &\left(\Htwo^{\graph_1\oslash\graph_2}\right)^2 = \frac{1}{2}\Tr\left[(w_1+w_2)^{-1}\right] = \left(\Htwo^{\graph_1}\right)^2 :\left(\Htwo^{\graph_2}\right)^2. 
  \end{align}
\end{proof}
We now propose the following Algorithm~\ref{alg:4} for computing the $\Htwo$ norm of a TTSP graph with control matrix $e_s$.
\removelatexerror
\begin{algorithm}[H]
  \label{alg:4}
  \SetAlgoNoLine
 \caption{$\Htwo$ norm of TTSP Graph with $B=e_i$}
\SetAlgoLined
\KwIn{Decomposition tree $\mathcal{T}(\graph)$, weights $\Weights(\graph)$, $e_i$}
\KwResult{ $(\Htwo^{\graph})^2$ }
 \For{each leaf ${\mathcal{L}}$ of $\mathcal{T}(\graph)$}{
  Output $(\Htwo^{\mathcal{L}})^2 = \frac{1}{2}w_{\mathcal{L}}^{-1}$ to parent\;
 }
 \For{each parent $j$ of $\mathcal{T}(\graph)$}{
  \eIf{received $\Htwo^{\graph_i}$ from both children}{
    \eIf{$j$ is a series join}{
       Output $(\Htwo)^2 = \left(\Htwo^{\graph_1}\right)^2 + \left(\Htwo^{\graph_2}\right)^2 $ to parent\;
    }{
       Output $(\Htwo)^2 = \left(\Htwo^{\graph_1}\right)^2 :\left(\Htwo^{\graph_2}\right)^2 $ to parent\;
    }
   }{
   wait\;
  }
 }
\Return{$(\Htwo)^2$ at root node of $\mathcal{T}(\graph)$.}
\end{algorithm}

We now proceed to the final result: 
\begin{theorem}
Consider a leader-follower consensus network on an all-input TTSP graph $\graph$.
Then, the $\Htwo$ norm is given by 
$
  (\Htwo^{\graph, B})^2 = \sum_{s\in \mathcal{R}}  (\Htwo^{\graph, e_s})^2,
$  
and the best-case complexity of computing this $\Htwo$ norm is $\mathcal{O}(|\mathcal{R}| \log |\Nodes|)$, and the worst-case complexity is $\mathcal{O}(|\mathcal{R}||\Nodes|)$.
\end{theorem}
\begin{proof}
  We can compute: 
  \begin{align}
    \left(\Htwo^{\graph, B}\right)^2  = - \frac{1}{2} \sum_{s\in\mathcal{R}} \Tr\left[ e_s^T A^{-1} e_s \right] = \sum_{s\in\mathcal{R}} \left[ \Htwo^{\graph_2,e_s} \right]^2.\label{eq:h2}
  \end{align}
  This is a sum of $|\mathcal{R}|$ $\Htwo$ norms of leader-follower consensus networks with control input $e_s$.
  Since at each layer of the decomposition tree $\mathcal{T}(\graph)$ each computation happens independently, the complexity depends on the height of the decomposition tree.
  The remainder of the proof is identical to that of Theorem~\ref{thm:complex}.
\end{proof}

\section{Conclusion}
\label{sec:conclusion}
In this paper, we used two-terminal series-parallel graphs for efficiently computing the $\Htwo$ performance for leader-follower consensus networks.
In particular, we computed the $\Htwo$ norm, as well as gradient updates to adapt the network for optimal $\Htwo$ performance, in best-case and worst-case complexity of $\mathcal{O}(|\mathcal{R}| \log |\Nodes|)$ and $\mathcal{O}(|\mathcal{R}||\Nodes|)$ respectively.

The authors would like to thank Airlie Chapman for many discussions on weight update schemes for leader-follower networks.


\end{document}